\documentclass[11pt]{amsart}

\usepackage{amsfonts, amstext, amsmath, amsthm, amscd, amssymb}

\usepackage{graphicx, color}

\usepackage{microtype}

\usepackage[hidelinks,pagebackref,pdftex]{hyperref}

\newtheorem{theorem}{Theorem}[section]
\newtheorem{proposition}[theorem]{Proposition}
\newtheorem{lemma}[theorem]{Lemma}
\newtheorem{corollary}[theorem]{Corollary}

\newtheorem*{namedtheorem}{\theoremname}
\newcommand{\theoremname}{testing}
\newenvironment{named}[1]{\renewcommand{\theoremname}{#1}\begin{namedtheorem}}{\end{namedtheorem}}

\theoremstyle{definition}
\newtheorem{definition}[theorem]{Definition}

\title[Hyperbolic Knots Given by Positive Braids]{Hyperbolic Knots given by Positive Braids with at Least two Full Twists}
\author{Thiago de Paiva}
\address[]{School of Mathematics, Monash University, VIC 3800, Australia }
\email[]{thiago.depaivasouza@monash.edu}

\begin{document}

\begin{abstract}
We give some conditions on positive braids with at least two full twists that ensure their closure is a hyperbolic knot, with applications to the geometric classification of T-links, arising from dynamics, and twisted torus knots. 
\end{abstract} 

\maketitle

\section{Introduction}
Thurston proved that any non-trivial knot in the 3-sphere is either a \emph{torus knot}, a \emph{satellite knot}, or has \emph{hyperbolic complement} \cite{Thurston}. These are called the \emph{geometric types} of the knot. A satellite knot is a knot whose complement has an essential torus; a torus knot is a knot whose complement has essential annulus but does not have any essential torus; finally, a hyperbolic knot is a knot that is neither a torus knot nor a satellite knot.  
Thus, the \emph{geometric classification} of knots is related to the study of certain types of surfaces in the knot complement in $S^3$.

The geometric classification of knots has been an important task in modern research in knot theory.
\emph{Alexander's theorem} tells us that every knot can be represented as the closure of a braid \cite{Birman}. So, a lot of study has been done to understand the knot types of knots given by braids. 
For example, Birman and Menasco studied surfaces in closed braid complements \cite{finiteness}. They also classified the positions of essential tori in closed braid complements \cite{positions}. 
Los proved that the Nielsen-Thurston classification of braids and the geometric structure of knot complements are related \cite{Los}. Ito studied the topology and geometry of closed braid complements by examining the Dehornoy ordering \cite{Ito}, which is a left-invariant total order on the braid group, introcuded by Dehornoy \cite{Dehornoy}. 
 
Some important families of knots have appeared naturally with projections given by braids, such as \emph{Lorenz links} and \emph{twisted torus knots}. Twisted torus knots were introduced by Dean in his doctoral thesis \cite{Thesis} to study
Seifert fibered spaces obtained by Dehn fillings.
Lorenz links are knotted closed periodic orbits in the flow of the Lorenz system. The Lorenz system is a system of three ordinary differential equations in $\mathbb{R}^3$ introduced by the meteorologist Edward Lorenz to predict weather patterns. Birman and Kofman proved that Lorenz links coincide with \emph{T-links} \cite{newtwis}, which are links given by certain positive braids.

There haven't been many concrete conditions on braids to ensure their geometric type. For example, not much is known about the geometric classification of T-knots.
However, there has been some study on the geometric classification of twisted torus knots(see \cite{LeeTorusknotsobtained}, \cite{hyperbolicity}, \cite{Composite}) of which some form a small class of T-knots.

This work builds upon the previous works by finding some conditions on some braids so that they close to give hyperbolic knots.

Let $\sigma_1, \dots, \sigma_{n-1}$ be the standard generators of the braid group $B_n$. 

\begin{definition}
Let $i, j, r$ be positive integers with $i<j$. The $(i, j, r)$-torus braid, denoted by $B_{i,j}^r$, is defined by the braid$$(\sigma_i\dots \sigma_{j-1})^{r}.$$
\end{definition}

Our main theorem is the following.

\begin{theorem}\label{theorem1}
Let $K$ be a knot given by a positive braid $B$ with $p$ strands. Suppose that $B$ has one $(1, p, q + pk)$-torus braid, where $p, q$  are positive coprime integers with $p>q$ and $q, k\geq 2$, but $B \neq B_{1, p}^{q + pk}$. Also, assume that the braid $B(\sigma_{p-1}^{-1}\dots \sigma_1^{-1})^{kp}$ has braid index equal to $q$. Then, the knot $K$ is hyperbolic.
\end{theorem}

To prove this result, we combine a result of Ito related to the positions of essential surfaces with respect to the braid axis of the closed braid with a result of Williams related to the braid index of generalized cablings.

This theorem has some interesting applications to the geometric classification of T-links and twisted torus knots. There is an important conjecture, which was proposed by Morton, related to the geometric classification of Lorenz knots(see \cite{dePaivaPurcell:SatellitesLorenz} for more details). It says the following: 
A Lorenz knot that is a satellite has companion a Lorenz knot. Its pattern, when embedded in an unknotted torus in $S^3$, is equivalent to a Lorenz knot in $S^3$. 
Theorem~\ref{theorem1} gives us the next corollary, which helps to reduce the cases we need to consider Morton's conjecture as well as to understand the geometric classification of T-knots.

\begin{corollary}\label{T-knot} 
Let $p, q$ be positive coprime integers. Consider $p>q\geq r_1>1$.  Then, for $k\geq 2$, the T-knots $$T((r_n, s_n), \dots, (r_1, s_1), (p, kp + q))\textrm{ and }T((r_n, s_n), \dots, (r_1, s_1), (kp + q, p))$$ are hyperbolic.  
\end{corollary}

The geometric classification of twisted torus knots has been intensely studied but the knot types of twisted torus knots of the form $T(p, q; r,\pm 1)$ are not fully understood(see \cite{dePaiva:Unexpected} for more details). We obtain the following two corollaries for these open cases. 

\begin{corollary}\label{positiveTTT} 
Let $p, q$ be positive coprime integers. Consider $p>q\geq r>1$.  Then, for $k\geq 2$, the twisted torus knots $$T(p, kp + q; r, 1)\textrm{ and }T(kp + q, p; r, 1)$$ are hyperbolic.  
\end{corollary}
\begin{corollary}\label{negativeTTT} 
Let $p, q$ be positive coprime integers with $p>q$. Consider $p-q\geq r>1$. Then, for $k\geq 3$, the twisted torus knots $$T(p, kp + q; r, -1)\textrm{ and }T(kp + q, p; r, -1)$$ are hyperbolic.  
\end{corollary} 

\subsection{Acknowledgment}I am grateful to Jessica Purcell, my supervisor, and Sangyop Lee for helpful discussions. 

\section{The last case of Williams' result}

In this section we'll find the braid indexes of generalized cablings over trivial knots. This case wasn't considered by Williams in \cite{Williams}, but it'll be helpful for this work.
 
Williams extended the notion of cable knots by defining generalized cabling in \cite{Williams} as follows:

\begin{definition}A
generalized $q$-cabling of a link $L$ we mean a link $L'$ contained in the
interior of a tubular neighbourhood $L\times D^2$ of $L$ such that
\begin{enumerate}
\item each fiber $D^2$ intersects $L'$ transversely in $q$ points;  and

\item all strands of $L'$ are oriented in the same direction as $L$ itself.
\end{enumerate} 
\end{definition}

He also proved the following theorem in the same paper.

\begin{theorem}\label{Williams}
The braid index is multiplicative under generalized cabling. In detail, if $L$ is a link with each
component a non-trivial knot and $L'$ is a generalized $q$-cabling of $L$
then $$\beta(L') = q\beta(L),$$ where $\beta(*)$ is the braid index of $*$.
\end{theorem}

Now we consider the case where $L$ is a trivial knot.

\begin{lemma}\label{lemma1}
Let $L'$ be a generalized $q$-cabling of the unknot $L$, with $L$ given by a positive braid with $n$ strands, where $n>1$. Also, assume the knot inside $L$ is given by a positive braid. Then, $L'$ has braid index equal to $q$.
\end{lemma}
\begin{figure}
\includegraphics[scale=0.58]{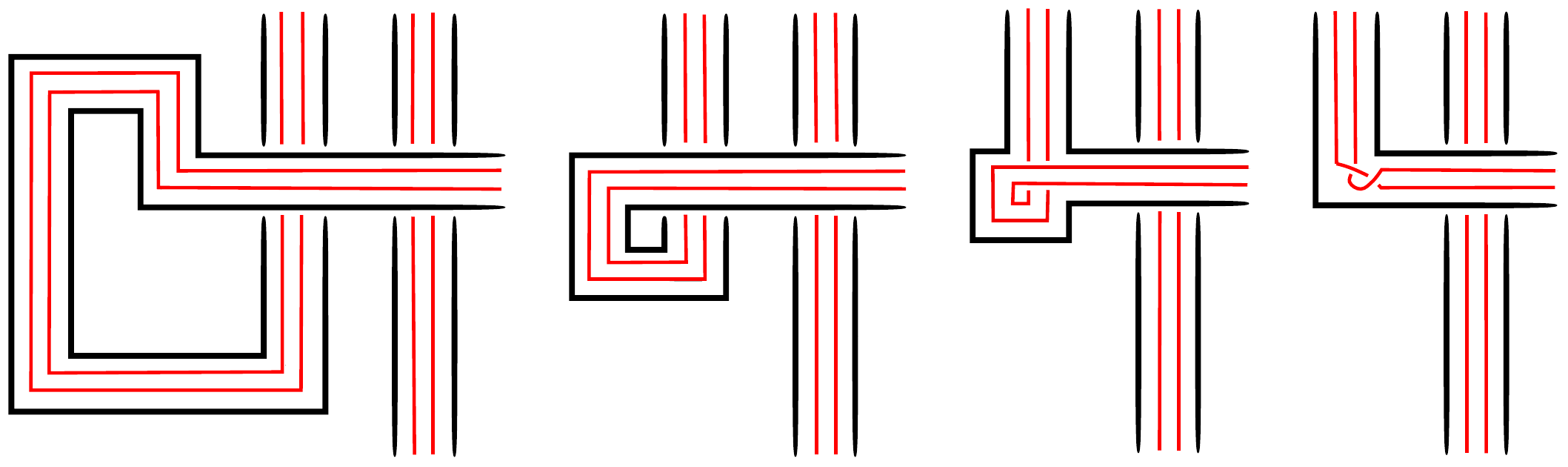} 
\caption{This series of drawings illustrate how the generalized 2-cabling $L'$ obtains one positive full twist after removing the crossing $\sigma_1$ from $L$. In this case we remove the first crossing of the braid $B$ of $L$. Similar ideas can be applied to see that $L'$ also obtains one positive full twist when we remove the last crossing of $B$ using the type II move. }
\label{trivial}
\end{figure}
\begin{proof} 
Denote by $B$ the positive braid of $L$. The Markov theorem says that we can apply two types of moves(type I and  type II)  to transform $B$ into the trivial braid with one strand \cite{Birman}. 
The type I  move geometrically sends crossings around the braid closure  to the top or bottom of the braid. Thus, the type I move doesn't change the knot inside $L$. The type II move removes strands which bound discs by shrinking and undoing the circles that they form. In addition, we see that the type II move adds a positive full twist to the knot inside $L$ as explained in Figure~\ref{trivial}. Since the braid $B$ has more than one strand, we need to apply at least one type II move to transform $B$ into the trivial braid with one strand. After all these Markov moves, the knot $L'$ is given by a positive braid with $q$ strands and at least one positive full twist on the $q$ strands. Now it follows from [\cite{Franks}, corollary 2.4] that the knot $L'$ has braid index equal to $q$.
\end{proof} 

\section{Half Twists: the ``generators" of the braid group}

In this section we define \emph{half twist} as in [\cite{dePaiva:Unexpected}, section 3]. We'll use them to calculate the braid indexes of some braids in the next section.

\begin{definition} 
Let $J$ be an unknot bounding a disc $D$ which transversely intersects a set of $j$ parallel straight segments in a plane
$Y$. A \emph{positive half twist along $J$} is obtained by the following
procedure: cut along $D$ into two discs $D_1$ and $D_2$, where $D_1$, $D_2$ is the top, bottom disc, respectively. Then, rotate $D_1$ by $180^\circ$ degrees in the clockwise direction and then glue it back to $D_2$. A \emph{negative half twist along $J$} is defined similarly, only the rotation is in the anti-clockwise direction. 
\end{definition}

A positive half twist transforms the trivial braid with $j$ strands into the braid
$$(\sigma_1\sigma_2\dots \sigma_{j - 2}\sigma_{j - 1})(\sigma_1\sigma_2\dots \sigma_{j - 3}\sigma_{j - 2})\dots (\sigma_1\sigma_2)(\sigma_1).$$
On the other hand, a negative half twist transforms the $j$ parallel straight segments into the braid 
$$(\sigma_{j - 1}^{-1}\sigma_{j - 2}^{-1} \dots \sigma_{2}^{-1}\sigma_{1}^{-1})(\sigma_{j - 1}^{-1}\sigma_{j - 2}^{-1} \dots \sigma_{2}^{-1})\dots (\sigma_{j - 1}^{-1}\sigma_{j - 2}^{-1})(\sigma_{j - 1}^{-1}),$$
where $\sigma_{i}^{-1}$ is the inverse of $\sigma_{i}$ in the braid group $B_j$.

We know that the $(i, j, (j-i+1)r)$-torus braid is obtained by Dehn filling along a circle encircling $j-i+1$ strands. The next lemma says that when the $(i, j, r)$-torus braid is not obtained by full twists, we need to use half twists in order to obtain it.

\begin{lemma}\label{half-twist}
The $(i, j, r)$-torus braid is obtained by full and half twists along three circles  as follows: start with the trivial braid $B$ with $j$ strands. Then, denote by $J_{i,j}$ the unknot encircling from the $i$ strand to the $j$ strand of the braid $B$. Consider $t$ an integer such that $0<t<j-i+1$ and $r=t+k(j-i+1)$ for some non-negative integer $k$. Further, augment the braid by placing unknots $J_{i, j-t}$ and $J_{j-t+1, j}$ into the trivial braid, where $J_{i, j-t}$, $J_{j-t+1, j}$ is enclosing from the strand $i$, $j-t+1$  to the strand $j-t$, $j$, respectively. Perform a positive half twist along $J_{i,j}$, a negative half twist along $J_{i, j-t}$, and a positive half twist along $J_{j-t+1, j}$. Finally, perform $(1/k)$-Dehn filling along $J_{i,j}$ and followed by $(1/0)$-Dehn fillings along each of $J_{i, j-t}$ and $J_{j-t+1, j}$ to remove them from the diagram.
\end{lemma}
 
\begin{proof}We start by applying a positive half twist along $J_{i,j}$ to obtain the braid
\[(\sigma_i\sigma_{i+1}\dots\sigma_{j-1})(\sigma_i\dots\sigma_{j-2})\dots(\sigma_i),\]
encircled by $J_{i,j}$. We push $J_{i,j}$ to lie above the braid. 
A negative half twist along $J_{i, j-t}$ yields the braid
\[(\sigma_{j-t-1}^{-1}\dots\sigma_{i+1}^{-1}\sigma_i^{-1})(\sigma_{j-t-1}^{-1}\dots\sigma_{i+1}^{-1})\dots(\sigma_{j-t-1}^{-1}).\]
Then, we slide the component $J_{i, j-t}$ to be below the braid of the resulting half twist. 
This negative half twist along $J_{i, j-t}$ is cancelled with the positive half twist along the first $j-t-i+1$ strands of the first braid, as illustrated in the first and second drawings of Figure \ref{HalfTwistBraid}.
Finally, the positive half twist along $J_{j-t+1, j}$ concatenates a positive half twist along the last $j-(j-t+1)+1$ strands, giving the braid
\[(\sigma_i\dots\sigma_{j-1})^{t},\]
as shown in the second and third drawings of Figure \ref{HalfTwistBraid}.
\begin{figure}
\includegraphics[scale=0.33]{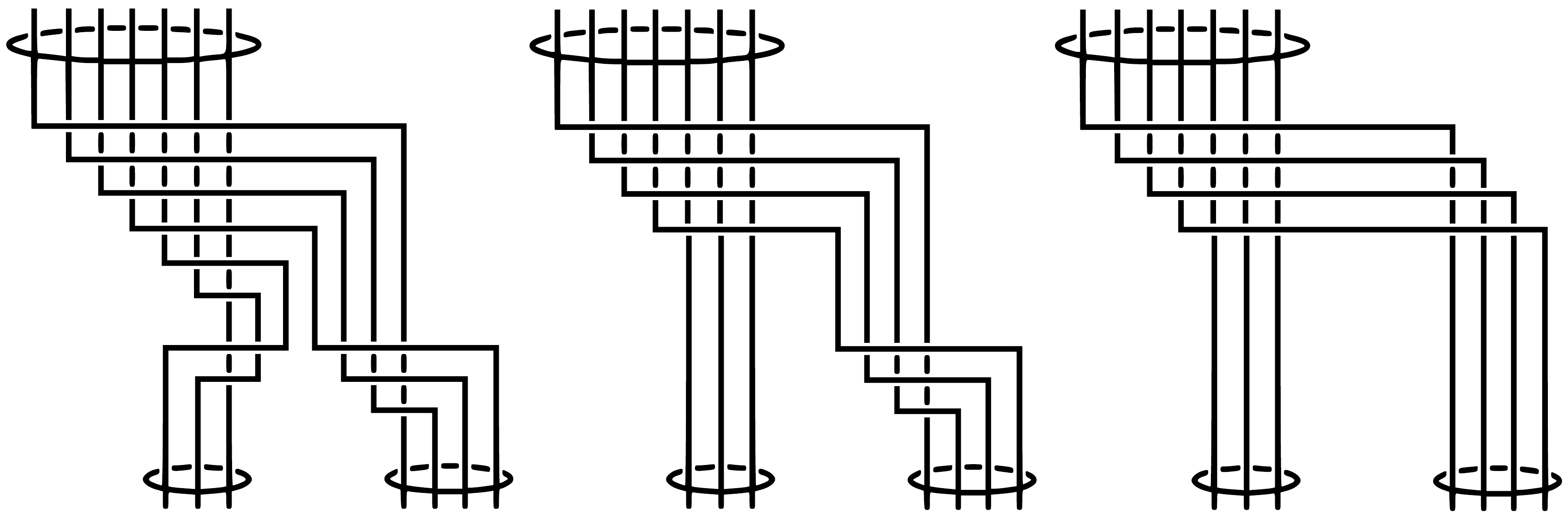} 
\caption{These drawings illustrate the procedure described in the proof of lemma~\ref{half-twist} to obtain the $(1, 7, 4)$-torus braid by half twists. 
We start with the trivial braid on 7 strands together with the circles $J_{1, 7}$, $J_{1, 3}$, and $J_{4, 7}$, where  $J_{1, 7}$ is encircling all 7 strands and $J_{1, 3}, J_{4, 7}$ is encircling the first 3, last 4 strands, 
respectively, with $J_{1, 7}$ above $J_{1, 3}$ and $J_{4, 7}$. Then, we apply a positive half twist along each $J_{1, 7}$ and $J_{4, 7}$ and a negative half twist along $J_{1, 3}$ to obtain the leftmost drawing. The negative half twist along $J_{1, 3}$ is cancelled with a  positive half twist along just the first $3$ strands as shown in the second drawing. Finally, the positive half twist along $J_{4, 7}$ joins the rest of the positive half twist along $J_{1, 7}$ to give the $(1, 7, 4)$-torus braid as illustrated by the last drawing.}
\label{HalfTwistBraid}
\end{figure}

Now we perform $(1/k)$-Dehn filling along $J_{i,j}$ to add an additional $k(j-i+1)$ overstrands($k$ full twists) into the braid to obtain the desired braid $B_{i, j}^r$, although it is still augmented by the unknots $J_{i, j-t}$ and $J_{j-t+1, j}$. To remove these additional components, we perform $(1/0)$-Dehn filling along each. 
\end{proof}

\section{Hyperbolic knots given by positive braids}

In this section we prove theorem~\ref{theorem1}.

\begin{proposition}\label{proposition1}
Let $K$ be a knot given by a positive braid $B$ with $p$ strands. Suppose that $B$ has one $(1, p, q + pk)$-torus braid, where $p, q$  are positive coprime integers with $p>q$ and $q, k\geq 2$. Also, assume that the braid $B(\sigma_{p-1}^{-1}\dots \sigma_1^{-1})^{kp}$ has braid index equal to $q$. Then, $S^3-K$ has no essential tori, and therefore $K$ is not a satellite knot.
\end{proposition}

\begin{proof}Consider that the complement of $K$ in $S^3$ has an essential torus $T$. Since $B$ has at least two positive full twists on $p$ strands, by Ito [\cite{Ito}, theorem 1.2(3)], $T$ doesn't intersect the braid axis $C$ of $B$ and the knot inside $T$ is given by a braid. So, $K$ is a generalized $b$-cabling of the knot $L$, where $L$ is the core of the solid torus bounded by $T$, with $b>1$. Thus, $b$ must divide $p$. 
After $(-1/k)$-Dehn surgery along $C$, the knot  $K$ becomes the knot $K'$ given by the closure of the braid $B(\sigma_{p-1}^{-1}\dots \sigma_1^{-1})^{kp}$ and the torus $T$ becomes a new torus $T'$ which doesn't intersect the braid axis of the last braid. The torus $T'$ is trivial or knotted. If $T'$ is knotted, then, by theorem~\ref{Williams}, $B(\sigma_{p-1}^{-1}\dots \sigma_1^{-1})^{kp}$ has braid index equal to $b\beta(T')$, where $\beta(T')$ is the braid index of the core of the solid torus bounded by $T'$. 
Thus, since the braid $B(\sigma_{p-1}^{-1}\dots \sigma_1^{-1})^{kp}$ has braid index equal to $q$, we have that $q = b\beta(T')$, which is a contradiction since $p$ and $q$ are coprime. If $T'$ is trivial, then, by lemma~\ref{lemma1}, $B(\sigma_{p-1}^{-1}\dots \sigma_1^{-1})^{kp}$ has braid index equal to $b$. Thus, $q = b$, a contradiction for the same reason. So, $T$ does not exist. Therefore, $S^3-K$ has no essential tori.
\end{proof}
 
\begin{corollary}
Let $K$ be a knot given by a positive braid $B$ with $p$ strands. Suppose that $B$ has one $(1, p, q + pk)$-torus braid, where $p, q$  are positive coprime integers with $p>q\geq 2$. Also, assume that the braid $B(\sigma_{p-1}^{-1}\dots \sigma_1^{-1})^{kp}$ has braid index equal to $q$. Then, for  $k = 1, 0$, if $S^3-K$ has an essential torus, it intersects the braid axis of the braid $B$. 
\end{corollary}
\begin{proof}
Suppose $S^3-K$ has an essential torus $T$ that does not intersect the braid axis $C$ of $B$. Then, by doing Dehn surgery along $C$ with large slope, we can assume that the closure of the braid $BB_{1, p}^{k'p}$ has an essential torus for at least one $k'>2$, which is a contradiction with the last proposition. 
\end{proof}

The next lemma was proved by Los [\cite{Los}, corollary 1.2]]. We'll use it in the next proposition and in corollary~\ref{negativeTTT}.
 
\begin{lemma}\label{Los}
Consider $\beta_1$, $\beta_2$ two braids which have minimal braid index and are representations of the same torus
knot. Suppose that these two closed braids travel around the same braid axis $C$. Then, there is an isotopy in the complement of $C$ that takes $\beta_1$ to $\beta_2$.
\end{lemma}

This lemma can also be seen by considering the Markov moves as mentioned in the proof of lemma~\ref{lemma1}. Since $\beta_1$, $\beta_2$ have minimal braid index, we only use type I moves to transform $\beta_i$ into $\beta_j$. Thus, this isopoty certainly happens in the complement of the braid axis.
 
\begin{proposition}\label{proposition2}
Let $K$ be a knot given by a positive braid $B$ with $p$ strands. Suppose that $B$ has one $(1, p, q + pk)$-torus braid, where $p, q$ are positive coprime integers with $p>q\geq 2$ and $k\geq 1$, but $B \neq B_{1, p}^{q + pk}$. Also, assume that the braid $B(\sigma_{p-1}^{-1}\dots \sigma_1^{-1})^{kp}$ has braid index equal to $q$. Then, the knot $K$ is not a torus knot.
\end{proposition}

\begin{proof}Consider that $K$ is a torus knot. Since the braid $B$ contains at least one full twist on
$p$ strands, it follows from [\cite{Franks}, corollary 2.4] that the knot $K$ has braid index equal to $p$. So, $K$ is the $(p, q + pk + d)$-torus knot with $d>0$.
After $(-1/k)$-Dehn surgery along the braid axis $C$ of $B$, the knot $K$ becomes the knot $K'$ given by the closure of the braid $B(\sigma_{p-1}^{-1}\dots \sigma_1^{-1})^{kp}$. 
So, the knot $K'$ is the $(p, q + d)$-torus knot because, by lemma~\ref{Los}, there is an isotopy in the complement of $C$ that takes $K$ to $T(p, q + pk + d)$. As $B(\sigma_{p-1}^{-1}\dots \sigma_1^{-1})^{kp}$ has braid index equal to $q$, $K$ is also the $(q, w)-$torus knot with $w>0$. So, $q$ is equal to $p$ or $q + d$. Since $p, q$ are coprime, $q$ is equal to $q + d$. Thus, $d=0$, a contradiction. Therefore, $K$ is not a torus knot.
\end{proof}

\begin{named}{Theorem~\ref{theorem1}}
Let $K$ be a knot given by a positive braid $B$ with $p$ strands. Suppose that $B$ has one $(1, p, q + pk)$-torus braid, where $p, q$  are positive coprime integers with $p>q$ and $q, k\geq 2$, but $B \neq B_{1, p}^{q + pk}$. Also, assume that the braid $B(\sigma_{p-1}^{-1}\dots \sigma_1^{-1})^{kp}$ has braid index equal to $q$. Then, the knot $K$ is hyperbolic.
\end{named}

\begin{proof}
It immediately follows from proposition~\ref{proposition1} and proposition~\ref{proposition2}.
\end{proof}

\begin{proposition}\label{proposition10}
Consider $p, q$ positive coprime integers with $p>q$ and $q, k\geq 2$. Let $K$ be a knot given by a positive braid $B$ with $p$ strands of the form
$$B_{a_1, b_1}^{r_1}B_{a_2, b_2}^{r_2}\dots B_{a_{m-1}, b_{m-1}}^{r_{m-1}}B_{0, p}^{q+kp}B_{a_{m+1}, b_{m+1}}^{r_{m+1}}\dots B_{a_n, b_n}^{r_n}$$ and different from $B_{0, p}^{q+kp}$
or with $q$ strands of the form
$$B_{a_1, b_1}^{r_1}B_{a_2, b_2}^{r_2}\dots B_{a_{m-1}, b_{m-1}}^{r_{m-1}}B_{0, q+kp}^{p}B_{a_{m+1}, b_{m+1}}^{r_{m+1}}\dots B_{a_n, b_n}^{r_n}$$ and different from $B_{0, q+kp}^{p}$
such that if  $a_i = 0$ then $b_i\leq q$ or if  $a_i \neq 0$ then $b_i-a_i+1\leq q$.
Then, the knot $K$ is hyperbolic.
\end{proposition}

\begin{proof}We will calculate the braid index of the knot $K'$ given by the closure of the following braid, denoted by $B'$,
$$B_{a_1, b_1}^{r_1}B_{a_2, b_2}^{r_2}\dots B_{a_{m-1}, b_{m-1}}^{r_{m-1}}B_{0, p}^{q}B_{a_{m+1}, b_{m+1}}^{r_{m+1}}\dots B_{a_n, b_n}^{r_n}.$$
\begin{figure}
\includegraphics[scale=0.4]{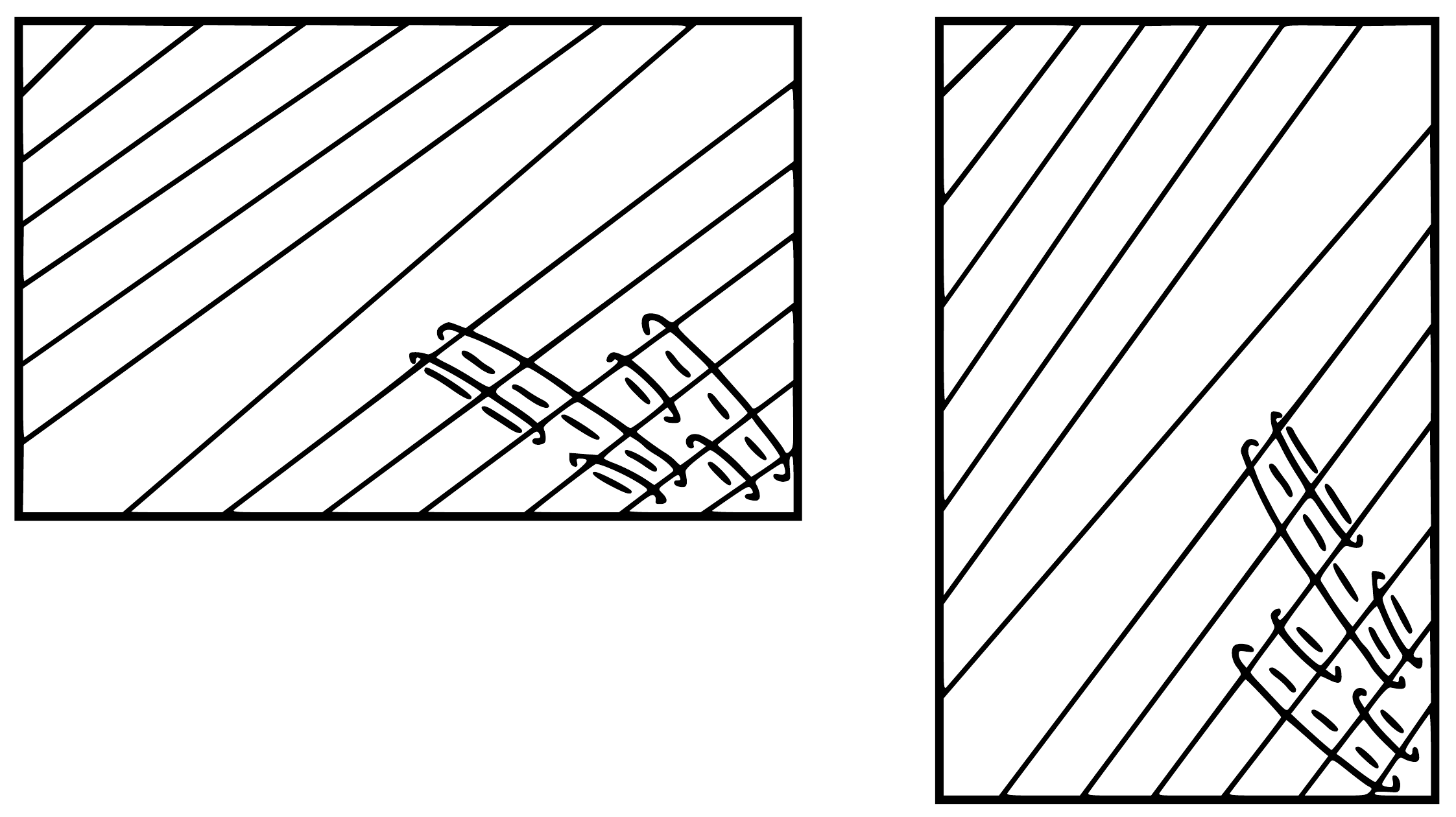} 
\caption{The second drawing is obtained from the first one by rotating it around the diagonal connecting the upper left vertex to the lower right vertex.}
\label{HK3}
\end{figure}

We push all $(a_i, b_i, r_i)$-torus braids which are above the biggest torus braid $B_{0, p}^{q}$  around the braid closure to be below  $B_{0, p}^{q}$. 

From lemma~\ref{half-twist}, each $(a_i, b_i, r_i)$-torus braid can be obtained by full and/or half twists along three circles where each one is encircling at most $q$ strands. Thus, $B'$ is obtained from the $(0, p, q)$-torus braid together with all these circles by full and/or half twists along them. 
By using the isotopy which deforms the $(p, q)$-torus knot to the $(q, p)$-torus knot, we can deform the $(0, p, q)$-torus braid to the $(0, q, p)$-torus braid. 
This isotopy can be described as a rotation around a diagonal of the square that gives the torus in which the $(p, q)-$torus knot lies, as illustrated in the first and second drawings of Figure~\ref{HK3}.
After this isotopy, the circles encircle the meridional(horizontal) lines of the $(0, q, p)$-torus braid. Furthermore, they enclose the same amount of strands as before, which are at most $q$ strands. 
Now we can push all these circles to lie below the $(0, q, p)$-torus braid and then apply the full and/or half twists as before.
We can see that after this last step we obtain a knot equivalent to the initial one(given by the closure of $B'$) because these isotopes preserve half twists as the tangles(bunches of crossings) of the torus braids travel in the same direction all the time.
The knot $K'$ is now given by a positive braid with the $(0, q, p)$-torus braid on the top. Since $p>q$, this braid has at least one full twist.  Therefore, it follows from \cite{Franks} that the knot $K'$ has braid index equal to $q$.
 
Now it follows from theorem~\ref{theorem1} that the knot given by the closure of the braid 
$$B_{a_1, b_1}^{r_1}B_{a_2, b_2}^{r_2}\dots B_{a_{m-1}, b_{m-1}}^{r_{m-1}}B_{0, p}^{q+kp}B_{a_{m+1}, b_{m+1}}^{r_{m+1}}\dots B_{a_n, b_n}^{r_n}$$ is hyperbolic since $B' = B(\sigma_{p-1}^{-1}\dots \sigma_{1}^{-1})^{kp}$.

With similar ideas, we can isotope the braid $$B_{a_1, b_1}^{r_1}B_{a_2, b_2}^{r_2}\dots B_{a_{m-1}, b_{m-1}}^{r_{m-1}}B_{0, q+kp}^{p}B_{a_{m+1}, b_{m+1}}^{r_{m+1}}\dots B_{a_n, b_n}^{r_n}$$
to a positive braid $B''$ with $p$ strands containing the torus braid $B_{0, p}^{q+kp}$. As before, we can see that the braid 
$B''(\sigma_1^{-1}\dots\sigma_{p-1}^{-1})^{kp}$ has braid index equal to $q$. Then, from theorem~\ref{theorem1}, 
the knot given by the closure of the braid 
$$B_{a_1, b_1}^{r_1}B_{a_2, b_2}^{r_2}\dots B_{a_{m-1}, b_{m-1}}^{r_{m-1}}B_{0, q+kp}^{p}B_{a_{m+1}, b_{m+1}}^{r_{m+1}}\dots B_{a_n, b_n}^{r_n}$$ is also hyperbolic.
\end{proof}

\section{Applications to the geometric classification of T-knots and twisted torus knots}
In this section we apply the results of the last section to get some corollaries related to the geometric classification of T-knots and twisted torus knots.

T-links are defined as follows:
for $2\leq r_1< \dots < r_k$, and all $s_i>0$, the T-link $T((r_1,s_1), \dots, (r_k,s_k))$ is defined to be the closure of the following braid
\[ (\sigma_1\sigma_2\dots\sigma_{r_1-1})^{s_1}(\sigma_1\sigma_2\dots\sigma_{r_2-1})^{s_2}\dots(\sigma_1\sigma_2\dots\sigma_{r_k-1})^{s_k}.\]
Birman and Kofman showed that they are equivalent to Lorenz links \cite{newtwis}.

The author and Purcell classified some T-knots obtained by full twists on a $(p, q)$-torus knot in \cite{dePaivaPurcell:SatellitesLorenz}. The following corollary also includes T-knots not obtained by full twists, which are more difficult to find their knot types since they are not obtained by Dehn filling.

\begin{named}{Corollary~\ref{T-knot}}
Let $p, q$ be positive coprime integers. Consider $p>q\geq r_1>1$.  Then, for $k\geq 2$, the T-knots $$T((r_n, s_n), \dots, (r_1, s_1), (p, kp + q))\textrm{ and }T((r_n, s_n), \dots, (r_1, s_1), (kp + q, p))$$ are hyperbolic.  
\end{named}

\begin{proof} 
It follows from proposition~\ref{proposition10} since the first T-knot is the closure of the braid $B_{0, r_n}^{s_n}\dots B_{0, r_{1}}^{s_{1}} B_{0, p}^{kp + q}$ and the second T-knot is the closure of the braid $B_{0, r_n}^{s_n}\dots B_{0, r_{1}}^{s_{1}} B_{0, kp + q}^{p}$.
\end{proof} 

The twisted torus knot $T(p, q; r, s)$ is obtained by twisting $r$ adjacent strands of the $(p, q)$-torus knot a total of $s$ full twists(see \cite{LeeThiago} for more details). If $r<p$, then the twisted torus knot $T(p, q; r, s)$ is given by the braid 
$$(\sigma_1\sigma_2\dots\sigma_{p-1})^{q}(\sigma_1\sigma_2\dots\sigma_{r-1})^{sr}$$ if $s>0$ and by $$(\sigma_1\sigma_2\dots\sigma_{p-1})^{q}(\sigma_{r-1}^{-1}\sigma_{r-2}^{-1}\dots\sigma_{1}^{-1})^{sr}$$ if $s<0$.

The following corollary addresses the conjecture 4.3 of \cite{dePaiva:Unexpected}.

\begin{named}{Corollary~\ref{positiveTTT}}
Let $p, q$ be positive coprime integers. Consider $p>q\geq r>1$.  Then, for $k\geq 2$, the twisted torus knots $$T(p, kp + q; r, 1)\textrm{ and }T(kp + q, p; r, 1)$$ are hyperbolic.  
\end{named}

\begin{proof} 
It follows from the last corollary that the twisted torus knot $T(p, kp + q; r, 1)$ is hyperbolic since $T(p, kp + q; r, 1) = T((r, r), (p, kp + q))$. From [\cite{Knottypes}, lemmas 1,3,4,5], the twisted torus knot $T(kp + q, p; r, 1)$ is equivalent to the twisted torus knot $T(p, kp + q; r, 1)$. Therefore, $T(kp + q, p; r, 1)$ is also hyperbolic.
\end{proof} 

The author and Lee classified some negative twisted torus knots which are torus knots in \cite{LeeThiago}. The next corollary classifies some negative twisted torus knots which are hyperbolic.

\begin{named}{Corollary~\ref{negativeTTT}}
Let $p, q$ be positive coprime integers with $p>q$. Consider $p-q\geq r>1$. Then, for $k\geq 3$, the twisted torus knots $$T(p, kp + q; r, -1)\textrm{ and }T(kp + q, p; r, -1)$$ are hyperbolic.  
\end{named}

\begin{proof} 
Suppose that $S^3-T(p, kp + q; r, -1)$ has an essential torus $T$. By Ito \cite{Ito}(see example 5.7), $T$ doesn't intersect the braid axis $C$ of $T(p, kp + q; r, -1)$. So, $T(p, kp + q; r, -1)$ is a generalized $b$-cabling of a knot $L$ with $b>0$, where $L$ is the core of the solid torus bounded by $T$. Thus, $b$ must divide $p$. 
After $(-1/(k+3))$-Dehn surgery along $C$, the knot  $T(p, kp + q; r, -1)$ becomes the knot $K'$ given by the closure of the braid $B_{0, p}^{2p + p -q}B_{0, r}^{r}$ and the torus $T$ becomes a new torus $T'$ which doesn't intersect the braid axis of the last braid. The torus $T$ can't be trivial otherwise $b = p$ by lemma~\ref{lemma1}, which is not possible as $T$ wouldn't be essential in $S^3-T(p, kp + q; r, -1)$. So, $T$ is knotted and thus essential in $S^3-K'$. But, by the last corollary, the twisted torus knot $T(p, 3p -q; r, 1)$, which is given by the closure of the braid $B_{0, p}^{3p -q}B_{0, r}^{r}$, is hyperbolic, contradiction. Therefore, $T$ does not exist.

Consider now that the twisted torus knot $T(p, kp + q; r, -1)$ is a torus knot. By sending the crossings of the braid $(\sigma_{r-1}^{-1}\dots \sigma_{1}^{-1})^{r}$ clockwise around the braid closure, they are cancelled with some crossings of the first $r$ horizontal lines of $(\sigma_{1}\dots \sigma_{p-1})^{kp + q}$. 
After that, the total braid becomes a positive braid and has at least two full twists on $p$ strands. Thus, it has braid index equal to $p$ \cite{Franks}. So, $T(p, kp + q; r, -1)$ is a $(p, d)$-torus knot with $d>0$. From lemma~\ref{Los}, there is an isotopy in the complement of $C$ that takes $T(p, kp + q; r, -1)$ to $T(p, d)$. Then, we do $(-1/(k+3))$-Dehn surgery along $C$ to transform $T(p, kp + q; r, -1)$ into $T(p, 3p-q; r, 1)$. 
But, then $T(p, 3p-q; r, 1)$ would also be a $(p, d')$-torus knot with $d'>0$. But, $T(p, 3p-q; r, 1)$ can't be a torus knot due to the existence of the last corollary. 

Finally, the twisted torus knot $T(kp + q, p; r, -1)$ is equivalent to the twisted torus knot $T(p, kp + q; r, -1)$ \cite{Knottypes}. Therefore, $T(kp + q, p; r, -1)$ is hyperbolic as well. 
\end{proof} 
\bibliographystyle{amsplain}  

\bibliography{Hyperbolicknotsgivenbypositivebraidswithatleasttwofulltwists}

\end{document}